\documentclass[12pt]{article}
\usepackage{amsmath,amssymb}
\usepackage{amsthm}

\usepackage{array}
\usepackage[matrix,arrow,curve]{xy}
\usepackage{color}
\usepackage{bbm}

\newcommand \trdeg {{\rm {tr.deg}}}

\newcommand\Spec{\operatorname{Spec}}

\newcommand\OO{{\mathcal O}}
\newcommand\Z{{\mathbb Z}}

\newcommand{\Cb}{{\mathbb C}}
\newcommand{\Pb}{{\mathbb P}}
\newcommand{\Qb}{{\mathbb Q}}
\newcommand{\Rb}{{\mathbb R}}

\theoremstyle{plain}
\newtheorem{theor}{Theorem}
\newtheorem{prop}[theor]{Proposition}

\newtheorem{lemma}[theor]{Lemma}

\theoremstyle{remark}
\newtheorem{rmk}[theor]{Remark}
\newtheorem{examp}[theor]{Example}

\theoremstyle{definition}

\date{}

\begin{document}

\title{Generation of modules and transcendence degree of zero-cycles}

\author{Sergey Gorchinskiy \\ \\
\small{Steklov Mathematical Institute, Moscow, Russia}\\
\small{e-mail: {\tt gorchins@mi.ras.ru}}}

\maketitle

\begin{abstract}
We construct an example of a regular algebra over $\Cb$ of dimension $d$ and a rank~$r$ projective module over it which is not generated by $d+r-1$ elements. This strengthens an example by Swan over the field of real numbers.
\end{abstract}

\bigskip \centerline{\it To I.\,R.\,Shafarevich, with a profound respect}
\bigskip

\section{Introduction}

Let $R$ be a unital finitely generated commutative algebra over a field $k$ and let $M$ be a finitely generated projective $R$-module of rank $r$. Let $d$ be the Krull dimension of $R$. It follows from the Forster--Swan theorem~\cite{For}, \cite{Swan} that $M$ is generated by $d+r$ elements over $R$ (actually, the theorem deals with rings and modules of a more general type). In our particular case this result can be also easily deduced from the Bertini theorem.

A natural question is wether the lawer bound $d+r$ on the number of generators is exact. In other words, wether there exist $R$ and $M$ as above such that $M$ is not generated by $d+r-1$ elements (the question is trivial for the case $d=1$). Swan~\cite[Th.4]{Swan1} has constructed such examples for all~$d$,~$r$ and \mbox{$k=\Rb$}. His argument essentially uses that the field $\Rb$ is not algebraically closed. Up to the author's knowledge, there was no an example in the literature with an algebraically closed field $k$. The aim of this paper is to construct examples of this type.

Using Chern classes with values in Chow groups, we reduce the initial problem to a question about a non-vanishing of elements in Chow groups of zero-cycles on affine varieties (this idea closely follows~\cite{Swan1}, where Stiefel--Whitney classes are used). Then we apply  transcendence degree of zero-cycles~\cite{GG} in order to get the non-vanishing. This approach is essentially based on the Bloch--Srinivas decomposition of diagonal~\cite{BS}.

Examples of algebras and modules as in this paper can be further used in order to construct new non-trivial examples of Jordan algebras with the help of a method from~\cite{Zh}. However, this requires that $R$ and $M$ satisfy certain additional conditions related to derivations. The actual construction of such examples remains an open problem.

This note arose from discussions with C.\,Shramov of a related question posed by V.\,N.\,Zhelyabin. The author is deeply grateful to both of them for a stimulating and friendly atmosphere during their common stay in Novosibirsk in Fall 2011. The author was partially supported by RFBR grant 11-01-00145-a,
MK-4881.2011.1, NSh grant 5139.2012.1, and by AG Laboratory HSE, RF gov. grant, ag. 11.G34.31.0023.

\section{Main result and reduction to zero-cycles}

Here is the main result of the paper.

\begin{theor}\label{theor-main}
For all natural numbers $d$ and $r$, there exist a regular finitely generated algebra $R$ over $\Cb$ of dimension $d$ and a projective finitely generated \mbox{$R$-module} $M$ of rank~$r$ such that $M$ is not generated over $R$ by $d+r-1$ elements.
\end{theor}

Let $k$ be again an arbitrary field, put $U:=\Spec(R)$, let $N$ be a rank one projective \mbox{$R$-module}, $M=N\oplus R^{\oplus (r-1)}$, and let $L$ be the line bundle on $U$ that corresponds to~$N$. For an algebraic variety $X$ over $k$, by $CH^p(X)$ denote the Chow group of codimension~$p$ cycles on $X$. Given a vector bundle $E$ on $X$, by \mbox{$c_p(E)\in CH^p(X)$} denote the corresponding Chern class. The following lemma is a direct analogue of an argument from~\cite[Ex.2]{Swan1}.

\begin{lemma}\label{lemma-selfint}
Suppose that $M$ is generated by $d+r-1$ elements. Then we have that $c_1(L)^d=0$ in~$CH^d(U)$.
\end{lemma}
\begin{proof}
By the condition of the lemma, there is an exact triple of vector bundles on $U$
$$
0\to E\to\OO_U^{\oplus {(d+r-1)}}\to L\oplus \OO_U^{\oplus(r-1)}\to 0\,.
$$
By the Whitney formula, we have
$$
1+c_1(E)+\ldots+c_d(E)=(1+c_1(L))^{-1}\,.
$$
In particular, $c_d(E)=(-1)^dc_1(L)^d$. On the other hand, since the rank of $E$ equals $d-1$, we obtain $c_d(E)=0$.
\end{proof}

Thus in order to construct our main example it would be helpful to show a non-vanishing of elements in $CH^d(U)$. It is important that the Chern classes here are with values in Chow groups and not in cohomology groups (Betti, de Rham, or \'etale), because degree $2d$ cohomology groups are trivial for any affine variety of dimension $d$.

\medskip

The following is an interpretation of the example~\cite[Ex.2,~Th.4]{Swan1} in terms of Chow groups.

\begin{examp}
Let $Q\subset\Pb^d$ be a smooth projective quadric without rational points over~$k$, $U=\Spec(R)$ be the complement to $Q$ in $\Pb^d$, $L=\OO_{\Pb^d}(1)|_U$, $N$ be the corresponding \mbox{$R$-module}, and let $M=N\oplus R^{\oplus(r-1)}$. Then $c_1(L)^d$ is the restriction to $U$ of the class of a point in $CH^d(\Pb^d)\cong\Z$. If $c_1(L)^d=0$, then it follows from the exact sequence
$$
CH^{d-1}(Q)\to CH^d(\Pb^d)\to CH^d(U)\to 0
$$
that the quadric $Q$ has a degree one zero-cycle over $k$. By a well-known result of Springer~\cite[Ch.VII,~Th.2.3]{Lam}, this contradicts with the absence of $k$-points on $Q$ (for $k=\Rb$, we also obtain this contradiction considering the action of the complex conjugation on an effective zero-cycle of odd degree). Therefore by Lemma~\ref{lemma-selfint},~$M$ is not generated by $d+r-1$ elements.
\end{examp}

\section{Transcendence of zero-cycles and the proof of the main result}

Let us recall some notions and facts from~\cite{GG}. Given a subfield $k_0\subset k$ and a variety~$X_0$ over $k_0$, by $(X_0)_k$ denote the extension of scalars of $X_0$ from $k_0$ to $k$:
$$(X_0)_k=\Spec(k)\times_{\Spec(k_0)}X_0\,.
$$
Let $X$ be a variety over $k$. For simplicity, assume that $X$ is irreducible. Let $d$ be the dimension of $X$. Suppose that $X$ is the extension of scalars of a variety~$X_0$ defined over a subfield $k_0\subset k$, i.e., it is fixed an isomorphism \mbox{$X\cong (X_0)_k$}. We say that an open (respectively, closed) subset in $X$ is defined over $k_0$ if it is the extension of scalars of an open (respectively, closed) subset in $X_0$.

A point $P\in X(k)=X_0(k)$ corresponds to a morphism of schemes \mbox{$P\colon\Spec(k)\to X_0$}. By $\xi_P$ denote its image and put
$$
\trdeg(P/k_0):=\trdeg(k_0(\xi_P)/k_0)\,.
$$
Explicitly, $\trdeg(P/k_0)$ is the transcendence degree over $k_0$ of the field generated by coordinates of $P$ with respect to an arbitrary affine open neighborhood defined over~$k_0$. Given an element $\alpha\in CH^d(X)$, its transcendence degree $\trdeg(\alpha/k_0)$ is the minimal natural number $n$ such that there exists a zero-cycle $\sum_i n_iP_i$ on $X$ over $k$ that represents~$\alpha$ and such that for any $i$, we have $\trdeg(P_i/k_0)\leqslant n$. One defines similarly transcendence degree for elements in $CH^d(X)_{\Qb}:=CH^d(X)\otimes_{\Z}\Qb$ (one considers all representatives with rational coefficients).

\begin{rmk}\label{rmk-bound}
In the above notation, assume that there is an open subset $U\subset X$ defined over $k_0$ such that $\alpha|_U=0$ in $CH^d(U)$. Then $\alpha$ is represented by a zero-cycle supported on $Z=X\smallsetminus U$. Since the dimension of $Z$ is strictly less than $d$ and $Z$ is defined over $k_0$, it follows from the definition of a transcendence degree of zero-cycles that $\trdeg(\alpha/k_0)<d$. This remains true when Chow groups are taken with rational coefficients.
\end{rmk}

The next statement follows directly from~\cite[Th.7]{GG}.

\begin{prop}\label{prop-trans}
Let $X$ be an irreducible smooth projective variety of dimension $d$ over a field of characteristic zero $k$. Assume that $X$ is the extension of scalars of a variety $X_0$ defined over a subfield $k_0\subset k$. Assume that any finitely generated field over $k_0$ can be embedded into $k$ over $k_0$ (equivalently, $k$ is algebraically closed and of infinite transcendence degree over $k_0$). Suppose that there is a point \mbox{$P\in X(k)$} such that $\trdeg(P/k_0)=d$ and the class~$[P]$ of $P$ in~$CH^d(X)_{\Qb}$ satisfies $\trdeg([P]/k_0)<d$. Then $H^0(X,\Omega^d_X)=0$.
\end{prop}

The idea of the proof of Proposition~\ref{prop-trans} is to consider $[P]$ as the restriction of the class of the diagonal $X_0\times X_0$ to the subscheme $X_0\times\Spec(k_0(X_0))$ and to use the Bloch--Srinivas decomposition of diagonal~\cite{BS}.

\medskip

Now we are ready to prove Theorem~\ref{theor-main}. Let $C$ be a smooth projective curve over $\bar\Qb$ of positive genus and put $X_0=C^d$, $X=(X_0)_{\Cb}$. Consider a complex point \mbox{$P=(P_1,\ldots,P_d)\in X(\Cb)$} that does not belong to any subvariety in~$X$ defined over $\bar\Qb$ except for $X$ itself. Such a point exists, because there are countably many subvarieties in~$X$ defined over $\bar\Qb$, while $\Cb$ is uncountable. Explicitly, $P$ satisfies the following condition: choose an affine open subset in $C$ defined over $\bar \Qb$ whose extension of scalars to $\Cb$ contains $P_1,\ldots,P_d$. One requires that the coordinates of $P_1,\ldots,P_d$ with respect to this affine chart generate a field of transcendence degree $d$ over $\bar\Qb$.

Further, consider the following divisors in $X$:
$$
D_i:=\{(x_1,\ldots,x_d)\in X\mid x_i=P_i\},\quad 1\leqslant i\leqslant d
$$
and a (reducible) divisor $D=\cup_i D_i$. Note that $D$ is not defined over $\bar \Qb$. Finally, let \mbox{$U=\Spec(R)$} be any non-empty affine open subset in~$X$ defined over $\bar\Qb$, $L=\OO_X(D)|_U$,~$N$~be the corresponding \mbox{$R$-module}, and let $M=N\oplus R^{\oplus(r-1)}$. Then we have
$$
c_1(L)^d=c_1(\OO_X(D))^d|_U=d!\cdot [P]|_U\in CH^d(U)\,.
$$
Suppose that $[P]|_U=0$ in $CH^d(U)_{\Qb}$. By Remark~\ref{rmk-bound}, we obtain $\trdeg([P]/k_0)<d$ in $CH^d(X)_{\Qb}$. By Proposition~\ref{prop-trans}, this contradicts with the condition
$$
H^0(X,\Omega_X^d)\cong H^0(C,\Omega_C^1)_{\Cb}^{\otimes d}\ne 0\,.
$$
Therefore, $[P]|_U\ne 0$ in $CH^d(U)_{\Qb}$. Consequently, $d!\cdot[P]|_U\ne 0$ in $CH^d(U)$ and by Lemma~\ref{lemma-selfint}, $M$ is not generated by $d+r-1$ elements over~$R$. This finishes the proof of Theorem~\ref{theor-main}.

\end{document}